\documentclass[12pt]{amsart}

\usepackage{amsmath}
\usepackage{amssymb}
\usepackage{amsfonts}
\usepackage{amsthm}
\usepackage{verbatim}
\usepackage{amscd}
\usepackage{cite}
\usepackage{leftidx}
\usepackage{enumerate}
\usepackage{txfonts}
\usepackage{mathpazo}
\usepackage[mathscr]{eucal}
\usepackage{hyperref}
 \DeclareMathOperator{\cw}{cw}
 \DeclareMathOperator{\dsc}{Dsc}
 \DeclareMathOperator{\syt}{syt}

\def\inv{^{-1}}
\def\that{\hat T}
\def\ainfty{A\langle\infty\rangle}

\DeclareMathOperator{\mgbar}{\overline\M_g}
\DeclareMathOperator{\del}{\partial}

\def\refp #1.{(\ref{#1})}

\newcommand{\A}{\mathcal{A}}

\newcommand{\M}{\mathcal{M}}

\newcommand{\oneover}[1]{\frac{1}{#1}}

\newcommand{\Cal}[1]{\mathcal #1}

\def\sbr #1.{^{[#1]}}
\def\sfl #1.{^{\lfloor #1\rfloor}}

\newcommand{\sbp}[1]{_{(#1)}}

\newcommand\subp [1] {_{(#1)}}

\def\inv{^{-1}}
\def\?{{\bf{??}}}

\def\M{\Cal M}
\def\A{\Bbb A}

\def\C{\mathbb C}
\def\P{\mathbb P}
\def\N{\mathbb N}
\def\R{\mathbb R}

\def\Spec{\text{\rm Spec} }

\def\Q{\mathbb Q}

\def\O{\mathcal O}

\def\g{\mathfrak g}

\def\1/2{\frac{1}{2}}

\def\I{\mathcal{ I}}

\def\2{{[2]}}
\def\l{\ell}
\def\nl{\newline}

\def\<{\langle}
\def\>{\rangle}

\def\2{{[2]}}
\def\l{\ell}

\def\scl #1.{^{\lceil#1\rceil}}
\def\spr #1.{^{(#1)}}
\def\sbc #1.{^{\{#1\}}}

\def\subpr#1.{_{(#1)}}

\def\beq{\begin{equation*}}
\def\eeq{\end{equation*}}

\def\g3{{\Gamma\spr 3.}}

\newcommand{\beql}[2]{\begin{equation}\label{#1}#2\end{equation}}

\newcommand{\eqspl}[2]{
\begin{equation}\label{#1}
\begin{split}
#2\end{split}\end{equation}}
\newcommand{\eqsp}[1]{\begin{equation*}
\begin{split}#1\end{split}\end{equation*}}

\newcommand{\beginalphaenum}{
\begin{enumerate}\renewcommand{\labelenumi}{ }
\item \begin{enumerate}
}

\def\eex{\end{rm}\end{example}}
\newcommand\newsection[1]{\section{#1}\setcounter{equation}{0}
}

\newcommand{\be}{\mathbb E}

\newtheorem{thm}{Theorem}[section]

\newtheorem*{thm*}{Theorem}
\newtheorem{cor}[thm]{Corollary}
\newtheorem*{cor*}{Corollary}

\newtheorem{lem}[thm]{Lemma}

\newtheorem*{claim*}{Claim}
\newtheorem{prop}[thm]{Proposition}

\theoremstyle{remark}

\newtheorem{rem}[thm]{Remark}

\newtheorem{example}[thm]{Example}
\newtheorem*{example*}{Example}

\begin{document}
\title{ \Large Differential equations in Hilbert-Mumford Calculus
}

\normalsize
\author 
{Ziv Ran}
\date {\today}
\address {Math Dept.  UC Riverside\nl
Surge Facility,  Big Springs Road,\nl
Riverside CA 92521}
\email {ziv.ran @ucr.edu} 
\begin{abstract}
An evolution-type differential equation encodes the intersection theory of tautological classes on the Hilbert scheme of a family of nodal curves.
\end{abstract}
\maketitle
\section*{Introduction}
Let $X/B$ be a family of nodal or smooth curves and $L$ a line bundle on $X$. Let $X\sbr m._B$ denote
the relative Hilbert scheme of length-$m$ subschemes of fibres of $X/B $, and $\Lambda_m(L)$ the tautological
bundle associated to $L$, which is a rank-$m$ bundle on $X\sbr m._B$.
The term 'Hilbert-Mumford Calculus' refers to the intersection calculus of 'tautological classes',
i.e. polynomials in the Chern classes of  $\Lambda_m(L)$. This calculus, which is an extension of the classical work of Macdonald
\cite{macd},  was developed in \cite{geonodal}, \cite{internodal} and
other papers, where a number of examples and computations were given, with
the more involved ones mostly based
on the Macnodal computer program developed for this purpose by Gwoho Liu \cite{macnodal}.
 Our purpose here is to show
that this calculus can be encoded in a linear second-order partial differential equation satisfied by a suitable generating function
(see \eqref{evolution}, \eqref{evolution-bis} below). While the result is, in a sense, just
a reformulation of results in \cite{internodal}, the advantages of the reformulation are that
it uses the standard language of differential calculus and moreover avoids the recursiveness inherent in
\cite{internodal}.\par
In more detail, let $W^m(X/B)$  denote the Hilbert scheme of length-$m$ flags in fibres and
consider the  infinite-flag Hilbert scheme \[
W(X/B)=\varprojlim W^m(X/B)\subset\prod_m X\sbr m._B\]
which is endowed with discriminant or big diagonal operators $\Gamma\spr m.$ pulled back from
$X\sbr m._B$ and with classes $L_i$ pulled back from the $i$-th $X$ factor.
It was shown in \cite{geonodal} that the Chern numbers of the tautological bundles can be expressed as linear
combinations of monomials of the form (working left to right)
\[L_1^{a_1}L_2^{a_2}(\Gamma\spr 2.)^{k_2}...L_r^{a_r}(\Gamma\spr r.)^{k_r}.\]
 Consequently we introduce the
'Hilbert potential'
\[G=\exp(\gamma\Gamma)\exp_\star(\sum\mu_i{L^i})\]
in which  $\star$ is external or 'Pontrjagin' product (whereas the 'implicit' or '.' product is intersection
or, in the case of an operator like $\Gamma$, composition).
Then the intersection calculus of \cite{internodal} shows how to express $G$
recursively in terms
of elements of the so-called tautological module $T=T(X/B)$,
and consequently how to read off numerical information.
We show in Theorem \ref{evo-thm} how to encode the latter into
an equation in  the $\gamma$- and $\mu_i$-derivatives of $G$ and its derivatives with respect
to the 'space' variables corresponding to standard generators of $T$. This
equation can be used to completely determine $G$.\par
In order to be able to express the appropriate relation in a familiar differential equation  form, we introduce a formal model $\that$ for the tautological
module $T$, essentially by replacing suitable generators by independent variables.\par
The use of differential equations to describe intersection theory associated to stable curves is not new.
Our evolution equation is somewhat  analogous to the 'quantum differential equation'
 of Gromov-Witten theory
(see \cite{cox-katz}, Ch. 10 or  \cite{mirror-book}, Ch. 28). Another well-known such equation is
Witten's KdV equation, governing the intersection theory of
the moduli space $\mgbar$ (see \cite{Witten}).
It would be interesting to find more direct connections.

\section{Big tautologocal module}

\subsection{Data}
We will fix a flat family $X/B$ of nodal, possibly pointed, genus-$g$ curves,
which is 'split' in the sense that its boundary can be covered by finitely many
projective families
of the form $X^\theta/B(\theta)\to X/B$, each endowed with a pair of distinguished sections $\theta_x, \theta_y$
called node preimages, that map to a node $\theta$ of $X/B$. We then have $i$-th
 boundary families
$X_i/B_i$ where \[B_i=\coprod\limits_{(\theta_1,...,\theta_i)}B(\theta_1,...,\theta_i)=
\coprod\limits_{(\theta_1,...,\theta_i)}B(\theta_1)\times_B...\times_BB(\theta_i)\] (union
over collections of $i$ distinct nodes). This includes the case $i=0$
where $B_0=B$.
To this we associate a \emph{coefficient system}, in the form of a system of pairs of graded unital $\Q$-algebras
\[(A_{B_i}\to A_i)=\bigoplus (A_{B(\theta_1,...,\theta_i)}\to A_{(\theta_1,...,\theta_i)})\]
such that\begin{enumerate}\item
$(A_{B_i}\to A_i)$ admits a map to $(H^*(B_i,\Q)\to H^*(X_i,\Q))$;\par
\item Each  $A=A_i$ contains an element $ \omega_i$ that maps to $c_1(\omega_{X_i/B_i})$, plus
elements that map to the distinguished sections, and each $A_{B_i}$ contains
elements mapping to Mumford classes and cotangent classes for the distinguished
sections (both those coming from $X/B$ and node preimages). There are all compatible, e.g.
\[\omega_i|_{X^{\theta_1,...,\theta_i}}=\omega+\sum\limits_{j=1}^i
(\theta_{j,x}+\theta_{j,y}).\]
\item For any distinguished section $\sigma$ over $B_i$,
there is a pullback map $\sigma^*:A_i\to A_{B_i}$.\par
\item  There are 'pullback' maps $(A_{B_i}\to A_i)\to (A_{B_{i+1}}\to A_{i+1})$ compatible with the various data.
An element $\alpha\in A_i$ may be replaced by its image in $A_j, j>i$, whenever this makes sense.
\end{enumerate}
\subsection{Generators, $\star$ product}
In \cite{internodal} we defined the tautotological module
\[T=T_A(X/B)=\bigoplus T^m_A(X/B).\]
This is graded by the weight $m$ which is the 'number of variables', i.e there is a canonical,
not necessarily injective, map
to the rational equivalence group
\[T^m\to A^\bullet_\Q(X\sbr m._B).\]
$T$ contains a 'classical' part $T_0$, which is a commutative algebra under
external or Pontrjagin product (as distinct from intersection product), which will be denoted by $\star$. Via the correspondence
\eqspl{}{
\begin{matrix}
&&W^{m+m'}(X/B)&\to&X\sbr m+m'._B\\
&\swarrow&&\searrow&\\
X\sbr m._B&&&&X\spr m'._B
\end{matrix}
}
$T$ is a module over $T_0$.
A special role will be played by the diagonal classes of $T_0$: the monoblock diagonals
\[\Gamma\sbp{n}[\alpha], \alpha\in A\]
and their $\star$- products, called polyblock diagonals.
In fact, if we introduce a formal variable $t_n, n\geq 1$,  we have a ring isomorphism
\[T_0\simeq A_B[t_nA: n\in\N ].\]
More concretely, $T_0$ is a direct sum of tensor products of symmetric powers of $A$ over $A_B$, indexed
by partitions.


%
In addition to polyblock diagonals, the tautological module also contains (iterated) node scrolls and node sections, of the form
\[F^n_j(\theta)[\gamma], Q^n_j(\theta)[\gamma], \gamma\in T_{A_\theta}(X^\theta/B(\theta))\]
(and their iterations).
Thus, elements of the tautological module of $X/B$ arise from analogous elements for
a boundary family $X^\theta/B(\theta)$ via a node scroll $F^n_J(\theta)$ or a node section $Q^n_j(\theta)$.
 To describe iterated node/scroll sections systematically,  let
$\theta_F, \theta_Q$ be mutually disjoint vectors of distinct nodes of $X/B$ of respective
dimensions $b_F, b_Q$, and
let $j_F, n_F, j_Q, n_Q$ be vectors  of natural numbers, indexed commonly with $\theta_F, \theta_Q$,
respectively.
Then we get
iterated node classes
\eqspl{fq-monomial}
{F^{n_F}_{j_F}(\theta_F)Q^{n_Q}_{j_Q}(\theta_Q)&[\prod_\star\Gamma\sbp{m_i}[\alpha_i]]\\=
...&F^{n_{F,i}}_{j_{F,i}}(\theta_{F,i})...Q^{n_{Q,i}}_{j_{Q,i}}(\theta_{Q,i})...[\prod_\star\Gamma\sbp{m_i}[\alpha_i]]\in
 T^{|m.|+b_F+b_Q}
}
Thus via $F^{n_F}_{j_F}(\theta_F)Q^{n_Q}_{j_Q}(\theta_Q)[*]$, we get a map
\[T_0(X^{(\theta_F\coprod \theta_Q)}/B(\theta_F\coprod\theta_Q))\to T(X/B).\]
$F^n_j$ and $Q^n_j$ are trivial unless $1\leq j<n$.
$F^{n_F}_{j_F}(\theta_F)Q^{n_Q}_{j_Q}(\theta_Q)[\Gamma^R]$  is the class of the \emph{closure} of the
part of  locus of type
$F^{R_F}_{j_F}(\theta_F)Q^{R_Q}_{j_Q}(\theta_Q)\star\Gamma^R$ where the points in the factor
corresponding to $\Gamma^R$ are in the smooth part of $X/B$. It coincides with the class of the latter
locus if either $R=0$ or $R_Q=0$, but differs from it otherwise. For example, the transfer formula of \cite{internodal}
reads, with this notation
\eqspl{star-gamma-1}
{F^n_j(\theta)\star\Gamma\sbp{1}[\alpha]=F^n_j(\theta)[\Gamma\sbp{1}[\alpha]],\\
Q^n_j(\theta)\star\Gamma\sbp{1}=Q^n_j(\theta)[\Gamma\sbp{1}[\alpha]]+\theta^*(\alpha)F^{n+1}_j(\theta)
}
In fact, a similar reasoning shows easily that
\[F^n_j(\theta)\star\Gamma\sbp{m}=F^n_j(\theta)[\Gamma\sbp{m}],\forall m\geq 1,\]
hence in fact
\eqspl{}{F^n_j(\theta)\star\prod_\star\Gamma\sbp{m_i}[\alpha_i]=F^n_j(\theta)[\prod_\star\Gamma\sbp{m_i}[\alpha_i]],}
The case of $Q^n_j$ is more involved: the relationship is the following
\begin{lem}\label{r-numbers-lem}
Define rational numbers $r(n,j)^k_\l$ for $0<j<n$ by
\eqspl{r(n,j)-eq}{
r(n,j)^j_{n+1}&=1;\\
r(n,j)^k_\l&=\frac{1}{\l-1}((\l-k)r^k_{\l -1}+kr(n,j)^{k-1}_{\l -1}), \l >n+1;\\
r(n,j)^*_*&=0, \mathrm{otherwise}.
}
and set
\eqspl{f(n,j)-eq}{
F(n,j,\theta,m,s)=\theta_x^*(s|_{A_\theta})\sum\limits_k r(n,j)^k_{n+m}F^{n+m}_k(\theta), s\in A
}
Then
\eqspl{}{
Q^n_j(\theta)\star\Gamma\sbp{m}[s]=Q^n_j(\theta)[\Gamma\sbp{m}[s]]+F(n,j,\theta,m,s)
}
\end{lem}
\begin{proof}
The case $m=1$ is just \eqref{star-gamma-1}. The general case is obtained by applying punctual transfer
 (cf. \cite{internodal}, \S 3.3)
$m-1$ times to the result of $\star\Gamma\sbp{1}[s]$, using \cite{internodal}, Prop. 3.19.
\end{proof}
\begin{rem}\label{r-numbers-remark}
Because $\theta_x, \theta_y$ both map to $\theta$, we have $\theta_x^*=\theta_y$.
Therefore we may denote both by $\theta^*$ and write the map $s\mapsto F(n,j,\theta, m,s)$ as
\[F(n,j,\theta, m)=\theta^*\sum\limits_k r(n,j)^k_{n+m}F^{n+m}_k(\theta).\]
Also, $\theta_x^*(\omega)=0$ (by residues), $\theta_x^*(1)=1$ (trivially).\qed
\end{rem}
The same argument  shows the following more general statement

\begin{prop}\label{star-[]-eq}
We have, with the above notations,
\eqspl{}{
&F^{n_F}_{j_F}(\theta_F)Q^{n_Q}_{j_Q}(\theta_Q)\star\Gamma\sbp{m_1}[s_1]\star...\star\Gamma\sbp{m_r}[s_r]=\\
&\ \
F^{n_F}_{j_F}(\theta_F)Q^{n_Q}_{j_Q}(\theta_Q)[\Gamma\sbp{m_1}[s_1]\star...\star\Gamma\sbp{m_r}[s_r]]\\
&+\sum\limits_{i,j}
F^{n_F}_{j_F}(\theta_F)F(n_{Q,i}, j_{Q,i}, \theta_{Q,i},m_j,s_j)
Q^{n_Q\setminus n_{Q,i}}_{j_{Q}\setminus j_{Q,i}}(\theta_Q\setminus \theta_{Q,i})
[\Gamma\sbp{m_1}[s_1]\star...\widehat{\Gamma\sbp{m_j}[s_j]}...\Gamma\sbp{m_r}]
}\qed
\end{prop}
Because $F$ classes are represented by $\P^1$-bundles, they automatically have vanishing integrals, so a nice simple consequence of Proporsition \ref{star-[]-eq} is
\begin{cor}\label{star-int} We have
\eqspl{}{
&\int F^{n_F}_{j_F}(\theta_F)Q^{n_Q}_{j_Q}(\theta_Q)\star\Gamma\sbp{m_1}[s_1]\star...\star\Gamma\sbp{m_r}[s_r]=\\
&\ \
\int F^{n_F}_{j_F}(\theta_F)Q^{n_Q}_{j_Q}(\theta_Q)[\Gamma\sbp{m_1}[s_1]\star...\star\Gamma\sbp{m_r}[s_r]].\qed 
}
\end{cor}

%
Note that the tautological module $T$ splits naturally as
\[T=\bigoplus\limits_{\theta.}T_{\theta.}\]
where the sum is over all vectors of distinct nodes and $T_{\theta.}$ consists of the classes that come from
the $\theta.$ boundary via a node scroll/section construction (though $T_{\theta.}$ is independent of the
ordering of $\theta.$, it is convenient to specify the ordering).
Thus
\eqspl{}{
T_{\theta.}=\bigoplus\limits_{\substack{\theta.=\theta_F\coprod \theta_Q,\\
 n_F,n_Q, j_F, j_Q}}F^{n_F}_{j_F}(\theta_F)Q^{n_Q}_{j_Q
}(\theta_Q)T_{A^{\theta.}}(X^{\theta.}/B(\theta.))
}
where $X^{\theta.}/B(\theta.)$ is the desingularized
boundary family corresponding to $\theta.$, endowed
with the node-preimage sections, and $A^{\theta.}$ is a coefficient ring on $X^{\theta.}/B(\theta.)$ as above.
\subsection{Standard model}\label{standard-model-sec}
We describe a standard model, actually just a notation change,
 for the tautological module $T$.
 This will be a free module $\hat T$ over a power series ring $\hat T_0$,
 in which $F^*_*(*), Q^*_*(*)$ and $\Gamma\sbp{*}[*]$ become variables  or formal symbols.
This will enable us to express the structure of $T$ in terms of standard operations such
 as differential operators.\par
 For each $n\geq 1$ let $t_n$ be a formal variable, let $t_0=1$,  and set
 \[ A\langle\infty\rangle=\bigoplus\limits_{n=0}^\infty At_n, \ainfty_{(\theta.)}=\bigoplus\limits_{n=0}^\infty A\sbp{\theta.}t_n\]
 as $A_B$ or $A_{B(\theta.)}$-module, respectively. Then we have an $A_B$-algebra
 \[\hat T_0=A_B[\ainfty].\]
 We think of generators $\alpha t_n\in \ainfty$ as corresponding to $\Gamma\sbp{n}[\alpha]$.and assign them weight $n$. Likewise,
 \[ \hat T_{0,(\theta.)}=A_{B(\theta.)}[\ainfty\sbp{\theta.}].\]
 We set
 \[\hat T_{0,*}=\bigoplus\limits_{(\theta.)}T_{0,(\theta.)}, \hat T_{0,i}=\bigoplus\limits_{|(\theta.)|=i}T_{0,(\theta.)}.\]
 For each node $\theta$, we designate formal  variables $\phi^n_j(\theta), \chi^n_j(\theta)$ corresponding
 to the node classes $F^n_j(\theta), Q^n_j(\theta)$.
 Now let $\theta_\phi, \theta_\chi$ be disjoint collections of distinct nodes, and let $n_\phi, j_\phi, n_\chi, j_\chi$
 be correspondingly-indexed vectors of natural numbers. Then set
 \[\phi^{n_\phi}_{j_{\phi}}(\theta_\phi)\chi^{n_\chi}_{j_\chi}(\theta_\chi)=\prod \limits_{(n_\phi,j_\phi, \theta_\phi)}
 \phi^n_j(\theta)\prod\limits_{
 (n_\chi, j_\chi, \theta_\chi)}\chi^n_j(\theta),\]
 \[ \hat T=\bigoplus \phi^{n_\phi}_{j_{\phi}}(\theta_\phi)\chi^{n_\chi}_{j_\chi}(\theta_\chi) \hat T_{0, \theta_\phi\coprod\theta_\chi}
 \]
 Thus, $\hat T$ is generated by symbols of the form \[\phi^{n_\phi}_{j_{\phi}}(\theta_\phi)\chi^{n_\chi}_{j_\chi}(\theta_\chi)\prod (t_{n_i}\alpha_i), \alpha_i\in A_{\theta_\phi\coprod\theta_\chi}\]
 and is a direct sum of $A_{(\theta.)}$ modules for the various collections $(\theta.)$ of distinct nodes.
 Moreover $\hat T$ is a $\hat T_{0,*}$-module.\par
 Note the map
 \eqspl{}{h:\hat T&\to T\\
 h(\phi^{n_\phi}_{j_{\phi}}(\theta_\phi)\chi^{n_\chi}_{j_\chi}(\theta_\chi)\prod (t_{n_i}\alpha_i))&=
F^{n_\phi}_{j_{\phi}}(\theta_\phi)Q^{n_\chi}_{j_\chi}(\theta_\chi)\star\prod_\star \Gamma\sbp{n_i}[\alpha_i]
}
$h$ is a bijection under which the $T_{0,*}$-module structure corresponds to $\star$ multiplication.\par
\begin{rem} \label{derivations}Note that for any $A_B$-linear map $\psi:A\to A$, there is a derivation $\psi t_n\del/\del t_n$ of $T$ defined by
\eqspl{derivation}{
\psi t_n\del/\del t_n (t_m\alpha)&=\begin{cases} \psi(\alpha), &m=n,\\ 0, &m\neq n;
\end{cases}
\\
\psi t_n\del/\del t_n (\phi^*_*(*)||\chi^*_*(*))&=0.
}
\end{rem}
Similarly, if $\psi:A\to A_{B_1}$ is an $A_B$-linear map, we can define a derivation
\eqspl{}{
\psi\phi^n_j(\theta)\del/\del t_n:\hat T\to\hat T,\\
\psi\phi^n_j(\theta)\del/\del t_n(t_m\alpha)=\begin{cases} \psi(\alpha)\phi^n_j(\theta), &m=n,\\ 0, &m\neq n;
\end{cases}
\\
\psi\phi^n_j(\theta) \del/\del t_n (\phi^*_*(*)||\chi^*_*(*))&=0.
}
\begin{rem}
Though not critical for our purposes, $\hat T$ can be made into a commutative associative ring under the proviso that
$\phi|\chi$ monomials must involve only distinct nodes $\theta$: i.e.
\[(\phi|\chi)^*_*(\theta)(\phi|\chi^*_*)(\theta)=0;\]
otherwise (i.e. where distinct $\theta$s are involved) $\phi$s and $\chi$s multiply formally.
\end{rem}

%

 \subsection{$\Gamma$ action}\label{gamma-action-sec}
For enumerative purposes, a crucial feature of $T$ is the
weight-graded action by the discriminant $\Gamma$. The nonclassical (boundary) part of the action is described
by the following rules.
\eqspl{}{
(\Gamma.{\prod}^\star \Gamma\sbp{n_i}[\alpha_i])_\theta=\sum\limits_{i}\sum\limits_{ 0<j<n_i}\frac{j(n_i-j)n_i}{2}F^{n_i}_j(\theta)
[{\prod\limits_{i'\neq i}}^\star \Gamma\sbp{n_{i'}}[\alpha_{i'} ]]
}
\eqspl{}{
-\Gamma.(F^n_j(\theta)[\gamma])=Q^n_j(\theta)[\gamma]+F^n_j[e^n_{j+1}.\gamma], \gamma\in T_{A^\theta}(X^\theta/B(\theta))\\
e^n_j(\theta)=-\Gamma_{X^\theta/B(\theta)}-(n-j+1)i(\theta_x)-ji(\theta_y)+\binom{n-j+1}{2}\psi_x(\theta)+\binom{j}{2}\psi_y(\theta)
}
\eqspl{}{
-\Gamma.Q^n_j(\theta)[\gamma]=Q^n_j[e^n_j(\theta).\gamma].
} Here $i(\theta_{x|y})$ refers to interior multiplication (see \S \ref{interior-multiplication-sec}).\par
The classical or interior part of the action of $\Gamma$ on $T_0$ is  described by
\eqspl{Gamma-0}{
\Gamma_0({\prod} ^\star\Gamma_{(n_j)}[\alpha_j])=\sum\limits_{j<j'}\Gamma_{(n_j+n_{j'})}[\alpha_j.\alpha_{j'}]\prod\limits_{k\neq j,j'}
\Gamma_{(n_k)}[\alpha_k]-\sum \binom{n_j}{2}\Gamma_{n_j}[\omega\alpha_j]
}

Note that $\Gamma_0$ has the nature of a second-order differential operator, in the following sense.
let $F$ be $\star$- polynomial in the $\Gamma_{(n)}[\alpha]$ with coefficients in $A_B$.  Let $\hat F=h\inv(F)\in \hat T$,
i.e. $\hat F$ is
the result of plugging in $t_{n}\alpha$ for each $\Gamma\sbp{n}[\alpha]$
 (and replacing $\star$ product by ordinary product). 
For $\alpha\in A$,  let $\alpha\del/\del t_n$ be the unique $A_B$-derivation on $\hat T_0$ such that
 \eqsp{
 \alpha\del/\del t_n(\alpha' t_{n'})=\begin{cases} \alpha\alpha', n'=n\\ 0, n'\neq n.
 \end{cases}
 } and of course $\del/\del t_n=1_A\del/\del t_n$.
 Then
\eqspl{}{
\Gamma_0F&=h(\hat \Gamma_0 \hat F), \mathrm{\ \ where}\\
\hat \Gamma_0:&=\sum\limits_{n\leq n'} nn't_{n+n'}\frac{\del^2}{\del t_{n}\del t_{n'}}
-\sum\limits_{n}\binom{n}{2}t_{n}\omega\frac{\del}{\del t_{n}}.
}
For example,
\[\hat\Gamma_0((t_n\alpha)( t_{n'}\alpha'))=
nn't_{n+n'}(\alpha.\alpha')-\binom{n}{2}t_n(\omega.\alpha)(t_{n'}\alpha')-\binom{n'}{2}(t_n\alpha)t_{n'}(\omega.\alpha'), n\neq n'.\]
This will be amplified below.\par
For later reference, we note the relation between $\Gamma_0$ on $T_0(X/B)$, as given by \eqref{Gamma-0}, and the corresponding operator on
$T_0(X^\theta/B(\theta))$ for a boundary family $X^\theta/B(\theta)$. The only difference is that $\omega=\omega_{X/B}$
is replaced by $\omega_{X^\theta/B(\theta)}=\omega(-\theta_x-\theta_y)$, where $\theta_x, \theta_y$
are the node preimage sections. Consequently, if we let $i\spr 2.$ be the derivation with respect to $\star$ product
 defined by
\eqspl{}{
{i}\spr 2.(\sigma)\Gamma\sbp{n}[\alpha]=\binom{n}{2}\Gamma\sbp{n}[\sigma.\alpha].
}
then we have
\eqspl{Gamma-0-comparison}{ \Gamma_{X^\theta/B(\theta), 0}= \Gamma_{0,X/B}|_{X^\theta}+ i\spr 2.(\theta_x+\theta_y)
}

\subsection{Interior multiplication}\label{interior-multiplication-sec}
Given any class $\alpha\in A$, there is an interior multiplication action $i(\alpha)$ on
the tautological module $T$: this is determined by the following conditions (where we
recall that a node $\theta$ is viewed as a map $B(\theta)\to X$ and yields
a pullback $\theta^*:A\to A_{B(\theta)}$):\begin{enumerate}
\item $i(\alpha)$ is a derivation with respect to $\star$ product;
\item $i(\alpha)\Gamma_{(n)}[\beta]=n\Gamma_{(n)}[\alpha.\beta]$;
\item $i(\alpha)F^n_j(\theta)[\beta]=F^n_j(\theta)[i(\alpha)\beta]+(\theta^*(\alpha))F^n_j(\theta)[\beta]$\item
$i(\alpha)Q^n_j(\theta)[\beta]=Q^n_j(\theta)[i(\alpha)\beta]+(\theta^*(\alpha))Q^n_j(\theta)[\beta]$.
\end{enumerate}
In applications, $\alpha$ will usually be a section (hence disjoint from
the node $\theta$),
so the second summand in the last two formulas it trivial. Therefore in such cases $i(\alpha)$ corresponds in
the model $\hat T$ to the operator \eqspl{delta}{
\delta(\alpha):=\sum \limits_{n}nt_{n}\alpha\del/\del t_{n}.
}
Similarly, the operator $i\spr 2.(\alpha)$ defined above corresponds to the derivation
\eqspl{delta-2}{
\delta\spr 2.(\alpha)=\sum\limits_{n}\binom{n}{2}t_{n}\alpha\del/\del t_{n}.
}
\subsection{$S$- transformation}\label{s-transformation-sec}
We seek a transformation on the tautological module taking $Q^n_j[\alpha]$ to $Q^n_j\star\alpha$. To this end, define rational numbers $r(n,j)^k_\l$
as in \eqref{f(n,j)-eq} (see Remark \ref{r-numbers-remark}).
Then set, as in  \eqref{derivation}
\eqspl{phi(n,j)-eq}{
\phi(n,j,\theta,m)=\sum\limits_k r(n,j)^k_{n+m}\phi^{n+m}_k(\theta)\theta^*
}
\eqspl{}{
\hat S=\sum \phi(n,j,\theta, m)\frac{\del^2}{\del \chi^n_j(\theta)\del t_{m}}
}
Then Proposition \ref{star-[]-eq} shows that $\hat S$ corresponds to an operator $S$ on $T$ such that
\[Q^n_j(\theta)[ \Gamma\sbp{m}[s]]=Q^n_j(\theta)\star[ \Gamma\sbp{m}[s]]-SQ^n_j(\theta)\star \Gamma\sbp{m}[s]\]
hence more generally
\eqspl{}{
F^{n_F}_{j_F}(\theta_F)Q^{n_Q}_{j_Q}(\theta_Q)[\prod_\star \Gamma\sbp{m.}[s.]]
=(I-S)\left(F^{n_F}_{j_F}(\theta_F)Q^{n_Q}_{j_Q}(\theta_Q)\star\prod_\star \Gamma\sbp{m.}[s.]\right).
}
Note that $\hat S^{b+1}=0$ where $b=\dim(B)$. Consequently,
\eqspl{I-S-inverse}{
(I-S)\inv=I+S+...+S^b.
}
\section{Evolution equation}
To introduce our evolution equation, we need some notation. First recall the 
corresponding to $\alpha\in A$ ( see \eqref{delta}, \eqref{delta-2}):
\eqspl{}{
\delta(\alpha)=\sum\limits_{n'}nt_{n}\alpha\frac{\del}{\del t_{n}}\\
\delta\spr 2.(\alpha)=\sum\limits_{n}\binom{n}{2}t_{n}\alpha\frac{\del}{\del t_{n}}.
}
 Then set
\eqspl{}{
\delta^n_j(\theta)=-(n-j+1)\delta(\theta_x)-j\delta(\theta_y)+\binom{n-j+1}{2}\psi_x(\theta)+\binom{j}{2}\psi_y(\theta)
} where the $\psi$ terms refer to the appropriate multiplication
operators. This is a first-order differential operator.\par
We will need to express the discriminant operator in terms of $\hat T$ with its $\hat T_0$-module structure, which will
involve rewriting terms like $Q^n_j(\theta)[\Gamma\sbp{n}]$ in terms of $Q^n_j(\theta)\star\Gamma\sbp{n}$. To this end, let
 $\hat \Gamma$ be the operator on $\hat T$ corresponding to $\Gamma$. It is
 $\hat\Gamma$ whose powers we wish to compute, as this will yields
 powers of $\Gamma$. The idea is to achieve that via a change of variable. Thus set, using
the notation of \S\ref{s-transformation-sec},
\eqspl{}{
\tilde \Gamma=(I-S)\inv\hat \Gamma(I-S)
.}
Then via $\hat\Gamma^k=(I-S){\tilde\Gamma}^k (I-S)\inv$, it suffices to compute powers of $\tilde\Gamma$. But
$\tilde \Gamma$ is a relatively 'elementary': specifically, a second-order differential operator.  
In  the above notations, we have, by a direct computation,
\eqspl{gamma-hat-eq}{
\tilde \Gamma=&\hat\Gamma_0+\sum \frac{j(n-j)n}{2}\theta_x^*\phi^n_j(\theta)\frac{\del}{\del t_{n}}-\sum \chi^n_j(\theta) \frac{\del}{\del \phi^n_j(\theta)}\\
& -\sum \phi^n_j(\theta)(\delta^n_{j+1}(\theta)- \delta\spr 2.(\theta_x+\theta_y))\frac{\del}{\del\phi^n_{j}(\theta)}+\chi^n_j(\theta)(\delta^n_j(\theta)-\delta\spr 2.(\theta_x+\theta_y))\frac{\del}{\del\chi^n_j(\theta)}
} where $\theta_x^*\phi^n_j(\theta)\frac{\del}{\del t_{n}}$ is as in Remark \ref{derivations}.
Here the $\delta\spr 2.(\theta_x+\theta_y)$ term comes from the difference between $\omega_{X/B}$
and $\omega_{X^\theta/B(\theta)}$.
Notice that because $S$ does not involve the $t$ variables, $\hat\Gamma_0$ coincides with the 'pure- $t$'
or classical portion of $\tilde\Gamma$.\par
Now we might consider the generating function $\exp(\gamma\tilde\Gamma)$ which encodes information about the powers
of the discriminant operator $\Gamma$ (weight unspecified).
%
%
%
As discussed in the Introduction, this is not sufficient for enumerative applications, which require monomials involving
discriminants of different weights and external multiplications. Fortunately the extension is not difficult to obtain.

To this end  let $\alpha_1,...,\alpha_r\in A$ be a set of homogeneous elements.
The results of \cite{geonodal} and \cite{internodal} show that Chern numbers of tautological
bundles $\Lambda_m(L)$, for a line bundle $L$ on $X$, 
on the flag-Hilbert schemes $W^m(X/B)$ of nodal curve families $X/B$  are given by linear combinations of monomials of the form
(read left to right)
\eqspl{monomials}{M=(\star\Gamma\sbp{1}[\alpha_1])(\star\Gamma\sbp{1}[\alpha_2])\Gamma^{k_2}...
(\star \Gamma\sbp{1}[\alpha_r])\Gamma^{k_r}
}
where $\alpha_i=L^{n_i}$.
Accordingly, we define, extending the above,
\eqspl{}{
G=\exp(\gamma\Gamma)\exp_\star(\sum \mu_i\Gamma\sbp{1}[\alpha_i])
\in T[[\gamma, \mu_1,...,\mu_r]],
} let
\eqspl{hat-G-eq}{
\hat G=\exp(\gamma\hat\Gamma)\exp(\sum\mu_i \alpha_i t_1)\in \hat T[[\gamma, \mu_1,...,\mu_r]]
}
be the corresponding element, and
\eqspl{tilde-G-eq}{
\tilde G=(I-S)\inv\hat G(I-S)
} (see \eqref{I-S-inverse}). Note that the first exponential in \eqref{hat-G-eq} refers to composition of operators
while the second refers to product in $\hat T_0$, which corresponds to $\star$ product.
We will use integral for an element of $\that$ to denote the integral of the corresponding element of $T$.

\begin{thm}\label{evo-thm}
The following differential equations hold:
\eqspl{evolution}{
&\del\tilde G/\del \gamma=\hat \Gamma_0\tilde G+\sum_{\theta, n,j} \frac{j(n-j)n}{2}\theta_x^*\phi^n_j(\theta)\del\tilde G/\del t_{n}
-\sum_{\theta, n,j} \chi^n_j(\theta)\del\tilde G/\del \phi^n_j(\theta)\\
& -\sum_{\theta, n,j} \phi^n_j(\theta)(\delta^n_{j+1}(\theta)-\delta\spr 2.(\theta_x+\theta_y))\del\tilde G/\del\phi^n_{j}(\theta)+\chi^n_j(\theta)(\delta^n_j(\theta)-\delta\spr 2.(\theta_x+\theta_y))\del\tilde G/\del\chi^n_j(\theta)
}

\eqspl{evolution-bis}{
\del\tilde G/\del \mu_i=t_{1}\alpha_i\tilde G+\sum_{\theta, n,j}\theta^*(\alpha_i)\phi^{n+1}_j(\theta)\del\tilde G/\del \chi^n_j(\theta).
}
Moreover,
\eqspl{integral-eq}{
\int\phi^{n_\phi}_{j_{\phi}}(\theta_\phi)\chi^{n_\chi}_{j_\chi}(\theta_\chi)\prod t_{n_i}\alpha_i=
\begin{cases} 0, n_\phi\neq\emptyset\\
\prod\int_X\alpha_i, n_\phi=\emptyset.
\end{cases}
}

\end{thm}
This Theorem, together with the obvious initial value $\tilde G(0,...,0)=1$ enables the computation
of $\tilde G$, hence of $G$, hence of monomials $M$ as in \eqref{monomials}.
\begin{proof}
To begin with, the first part of  relation \eqref{integral-eq} is essentially obvious , as $\phi$ variables correspond to $\P^1$-bundles of type $F$. The second part follows from Corollary \ref{star-int}, as
$\chi$ variables correspond to sections of type $Q$ of the $F$-bundles, and as far as integrals are concerned, $Q[\alpha]$ is equivalent
to $Q\star\alpha$.\par
Now the relation \eqref{evolution} encapsulates the computation of the $\Gamma$ operator as carried out in \cite{internodal}, \S 2. Schematically, applying $(-\Gamma)$ to a class of the form $F[y]$, $F=F^n_j(\theta)$,  yields the sum of
\par (i) the corresponding $Q[y]$ class;\par 
(ii) a class $F[dy]$ where $d$ is analogous to $\delta^n_{j+1}(\theta)$ above;\par 
(iii) the class $F[-\Gamma y]$.\par
Applying $-\Gamma$ to $Q[y]$ yields a sum of only the last two types (with $j$ in place of $j+1$).\par 
The first and second terms on the right of \eqref{evolution} correspond to the interior and boundary part of applying $\Gamma$ to polyblock diagonals and generally to the polyblock factor of an $FQ$- monomial as in
\eqref{fq-monomial}
(see \cite{internodal}, Thm. 2.23 ).
The third  term represents item (i) above for 
the action of $\Gamma$ on each $F$ factor.
In the final summation, the $\phi\delta$ term represents item (ii) above for each $F$, while the
$\chi\delta$ term represents the corresponding term for each $Q$
 (see \cite{internodal}, Theorem 2.24 and Remark 2.26 ).
The $\delta\spr 2.$ term are the result of '$\omega$ adjustment' as in \eqref{Gamma-0-comparison}, i.e writing
\[\omega_{X^\theta/B(\theta)}=\omega_{X/B}\otimes\O_{X^\theta}(-\theta_x-\theta_y).\]
Because different nodes $\theta$ are disjoint, no products of $\theta$-s appear.
\par
Equation \eqref{evolution-bis} is a consequence of the Transfer Theorem of \cite{internodal} (see Theorem 3.4 and display
(3.1.19)). The second term is a reflection of the $F^{n+1}_j(\theta)$ term in the transfer of $Q^n_j(\theta)$.

\end{proof}

\bibliographystyle{amsplain}
\bibliography{mybib}
\end{document}